\documentclass[11pt,a4paper,reqno]{amsart}
\usepackage[czech,english]{babel}
\usepackage{amscd,amssymb}
\usepackage{dsfont}
\usepackage{url}
\usepackage[all]{xy}
\usepackage{hyperref}
\title{$\clP_1$-covers over commutative rings}
\author[S.~Bazzoni]{Silvana Bazzoni}
\address[Silvana Bazzoni]{%
Dipartimento di Matematica ``Tullio Levi-Civita'' \\
Universit\`a di Padova \\
Via Trieste 63, 35121 Padova (Italy)}

\email{bazzoni@math.unipd.it}
\author[G.~Le Gros]{Giovanna Le Gros}
\address[Giovanna Le Gros]{%
Dipartimento di Matematica ``Tullio Levi-Civita'' \\
Universit\`a di Padova \\
Via Trieste 63, 35121 Padova (Italy)}

\email{giovannagiulia.legros@math.unipd.it}

\thanks{Research supported by grants from Ministero dell'Istruzione, dell'Universit\`a e della Ricerca (PRIN: ``Categories, Algebras: Ring-Theoretical and Homological Approaches (CARTHA)'') and Dipartimento di Matematica ``Tullio Levi-Civita'' of Universit\`a di Padova (Research program DOR1828909 ``Anelli e categorie di moduli''). }

\renewcommand{\iff}{if and only if }

\newcommand{\m}{\mathfrak{m}}

\newcommand{\p}{\mathfrak{p}}

\newcommand{\Hom}{\operatorname{Hom}}

\newcommand{\Ext}{\operatorname{Ext}}
\newcommand{\Tor}{\operatorname{Tor}}

\newcommand{\Ker}{\operatorname{Ker}}

\newcommand{\Coker}{\operatorname{Coker}}

\newcommand{\wdim}{\operatorname{w.dim}}
\newcommand{\pdim}{\operatorname{p.dim}}

\newcommand{\Fdim}{\operatorname{F.dim}}
\newcommand{\Fwdim}{\operatorname{F.w.dim}}
\newcommand{\fdim}{\operatorname{f.dim}}

\newcommand{\id}{\operatorname{id}}

\newcommand{\A}{\mathcal{A}}
\newcommand{\B}{\mathcal{B}}
\newcommand{\C}{\mathcal{C}}
\newcommand{\D}{\mathcal{D}}

\newcommand{\F}{\mathcal{F}}

\newcommand{\clP}{\mathcal{P}}

\newcommand{\clS}{\mathcal{S}}

\newcommand{\T}{\mathcal{T}}

\newcommand{\Modr}[1]{\mathrm{Mod}\textrm{-}{#1}}

\newcommand{\modr}[1]{\mathrm{mod}\textrm{-}{#1}}

\newcommand{\ModR}{\mathrm{Mod}\textrm{-}R}
\newcommand{\RMod}{R\textrm{-}\mathrm{Mod}}

\newcommand{\ModQ}{\mathrm{Mod}\textrm{-}Q}

\newcommand{\QMod}{Q\textrm{-}\mathrm{Mod}}
\newcommand{\modR}{\mathrm{mod}\textrm{-}R}
\newcommand{\modQ}{\mathrm{mod}\textrm{-}Q}

\newcommand{\clPmR}{\clP_1(\mathrm{mod}\textrm{-}R)}
\newcommand{\clPmQ}{\clP_1(\mathrm{mod}\textrm{-}Q)}

\theoremstyle{plain}
\newtheorem{thm}{Theorem}[section]
\newtheorem{lem}[thm]{Lemma}
\newtheorem{prop}[thm]{Proposition}
\newtheorem{cor}[thm]{Corollary}

\theoremstyle{definition}
\newtheorem{defn}[thm]{Definition}

\newtheorem{expl}[thm]{Example}

\theoremstyle{remark}
\newtheorem{rem}[thm]{Remark}

\begin{document}
\begin{abstract}
In this paper we consider the class $\clP_1(R)$ of modules of projective dimension at most one over a commutative ring $R$ and we investigate when $\clP_1(R)$ is a covering class. More precisely, we investigate Enochs' Conjecture for this class, that is the question of whether $\clP_1(R)$ is covering necessarily implies that $\clP_1(R)$ is closed under direct limits. We answer the question affirmatively in the case of a commutative semihereditary ring $R$. This gives an example of a cotorsion pair $(\clP_1(R), \clP_1(R)^\perp)$ which is not necessarily of finite type such that $\clP_1(R)$ satisfies Enochs' Conjecture. Moreover, we describe the class $\varinjlim \clP_1(R)$ over (not-necessarily commutative) rings which admit a classical ring of quotients.

\end{abstract}

\maketitle
\section{Introduction}
Approximation theory in module categories was studied in the setting of finite dimensional algebras by Auslander, Reiten, and Smal\o \ and independently by Enochs and Xu for modules over arbitrary rings using the terminology of preenvelopes and precovers.

An important problem in approximation theory is when minimal approximations, that is covers or envelopes, exist. In other words, for a certain class $\C$, the aim is to characterise the rings over which every module has a minimal approximation provided by $\C$ and furthermore to characterise the class $\C$ itself. 
 Bass proved in~\cite{Bass} that projective covers rarely exist, that is he introduced and characterised the class of perfect rings which are exactly the rings over which every module admits a projective cover. This motivated the study of minimal approximations for an arbitrary class $\C$.

Among the many characterisations of perfect rings, the most important from the homological point of view is the closure under direct limits of the class of projective modules. A famous theorem of Enochs says that for a class $\C$ in $\ModR$, if $\C$ is closed under direct limits, then any module that has a $\C$-precover has a $\C$-cover \cite{Eno}, \cite[Theorem 2.2.6 and 2.2.8]{Xu}. 

The converse problem, that is the question of when $\C$ is covering implies that $\C$ is closed under direct limits, is still an open problem which is known as Enochs' Conjecture.

In 2018, Angeleri H\"ugel-\v Saroch-Trlifaj in \cite{AST17} proved that Enochs' Conjecture holds for a large collection of left-hand classes of cotorsion pairs. Explicity, they proved that for a cotorsion pair $(\mathcal{A}, \mathcal{B})$ such that $\mathcal{B}$ is closed under direct limits, $\mathcal{A}$ is covering if and only if it is closed under direct limits. To prove this, Angeleri H\"ugel-\v Saroch-Trlifaj used methods developed in \v{S}aroch's paper \cite{S17}, which uses sophisticated set-theoretical methods in homological algebra.

  In this paper we are interested in Enochs' Conjecture for the class $\clP_1(R)$. The question naturally splits into two cases: the case that  the cotorsion pair $(\clP_1(R), \clP_1(R)^\perp)$ is of finite type (which occurs if and only if $\clP_1(R)^\perp$ is closed under direct sums, or equivalently when $(\clP_1(R), \clP_1(R)^\perp)$ is a $1$-tilting cotorsion pair), and the case when it is not of finite type. 
  
 In a forthcoming paper \cite{BLG20} we consider $1$-tilting cotorsion pairs $(\A, \B)$ over commutative rings $R$ and characterise the rings over which $\A$ is covering using a purely algebraic approach.

In this paper we consider the case that the cotorsion pair $(\clP_1(R), \clP_1(R)^\perp)$ is not necessarily of finite type (see Proposition~\ref{P:not-finite-type}). To the best of our knowledge, up until now there are no positive results for this question. Thus this paper provides a first positive result in the case of non-finite type. 

In the investigation of when $\clP_1(R)$ is covering, the class $\varinjlim \clP_1(R)$ plays an important role, although it is not always well understood. Unlike the case of the projective modules, where their direct limit closure is the class of flat modules, it is not necessarily true that the direct limit closure of $ \clP_1(R)$ coincides with the class $\F_1(R)$, of the modules of weak dimension at most  $1$. The inclusion $\varinjlim \clP_1(R) \subseteq \F_1(R)$ always holds, however an example of rings where $\varinjlim \clP_1(R) \subsetneq \F_1(R)$ can be found in \cite[Example 9.12]{GT12}. For certain nice rings, such as commutative domains, the two classes $\varinjlim \clP_1(R)$ and $\F_1(R)$ coincide (\cite[Theorem 9.10]{GT12}).

This paper is structured as follows. 
We begin in Section~\ref{S:prelim} with some preliminaries.

The aim of Section~\ref{S:limP_1} is to give a characterisation of the class $\varinjlim \clP_1(R)$ for a not-necessarily commutative ring $R$ which has a classical ring of quotients $Q$. This generalises a result from \cite{BH09} which was proved under the additional assumption that the little finitistic dimension of $Q$ is zero. A main result of Section~\ref{S:limP_1} is Proposition~\ref{P:P1-F1}, which states that $\varinjlim \clP_1(R)$ is exactly the intersection of $\F_1(R)$ with the left $\Tor^R_1$-orthogonal of the minimal cotilting class of cofinite type of $\QMod$ (see the definitions in Section~\ref{S:limP_1}).
%

After an overview of some useful results for commutative rings in Section~\ref{S:properties}, in Section~\ref{S:sh-P1-cov} we
assume that $\clP_1(R)$ is a covering class, and state some consequences of this assumption for the total ring of quotients $Q$ of $R$ and for localisations of $R$.



Finally in Section~\ref{S:sh-comm} we restrict to looking at only commutative semihereditary rings.
The main result of this paper is a positive solution of Enochs' Conjecture for the class $\clP_1(R)$ over a commutative semihereditary ring $R$. In Theorem~\ref{T:sh-P1-cov-lim} we show that in this case $\clP_1(R)$ is covering if and only if the ring is hereditary, which clearly implies that $\clP_1(R)$ is closed under direct limits. 
This provides us with an example of a class of rings for which $\clP_1(R)$ satisfies Enochs' Conjecture even though $(\clP_1(R), \clP_1(R)^\perp)$ may not be of finite type. 

 \section{Preliminaries}\label{S:prelim}
  $R$ will always denote an associative ring with unit and $\Modr R$ ($\RMod$) the category of right (left) $R$-modules.

For a ring $R$,  $\modr R$ will denote the class of right $R$-modules admitting a projective resolution consisting of finitely generated projective modules.  

Let $\C $ be a class of right $R$-modules. The right $\Ext^1_R$-orthogonal and right $\Ext^\infty_R$-orthogonal classes of $\C$ are defined as follows. 

\[
\C ^{\perp_1} =\{M\in \Modr R \ | \ \Ext_R^1(C,M)=0  \ {\rm for \
all\ } C\in \C\}\]
\[\C^\perp = \{M\in \Modr R \ | \ \Ext_R^i(C,M)=0  \ {\rm for \
all\ } C\in \C, \ {\rm for \
all\ } i\geq 1 \}\]

The left Ext-orthogonal classes ${}^{\perp_1} \C$ and ${}^\perp \C$ are defined symmetrically. 

For $\C$ a class in $\ModR$, the right $\Tor^R_1$-orthogonal and right $\Tor^R_\infty$-orthogonal classes are classes in $\RMod$ defined as follows. 

\[\C^{\intercal_1}=\{M\in \RMod  \ | \ \Tor^R_1(C,M)=0, \ {\rm for \
all\ } C\in \C  \} 
\]
\[\C^{\intercal}=\{M\in \RMod  \ | \ \Tor^R_i(C,M)=0 \ {\rm for \
all\ } C\in \C,  \ {\rm for \ all\ } i \geq1\}\]

The left $\Tor^R_1$-orthogonal and left $\Tor^R_\infty$-orthogonal classes ${}^{\intercal_1}\C$, ${}^{\intercal}\C$ are classes in $\ModR$ which are defined symmetrically for a class $\C$ in $\RMod$. 
%

If the class $\C$ has only one element, say $\C = \{X\}$, we write $X^{\perp_1}$ instead of $\{X\}^{\perp_1}$, and similarly for the other $\Ext$-orthogonal and $\Tor$-orthogonal classes.
\\

We denote by $\clP_n(R)$, ($\F_n(R)$) the class of right $R$-modules of projective (flat) dimension at most $n$ and by $\clPmR$ the class $\clP_1(R)\cap \modR$, that is the class of finitely presented right $R$-modules of projective dimension at most $1$. 

 The projective dimension (weak or flat dimension) of a right $R$-module $M$ is denoted $\pdim_R M$ ($\wdim_R M$).
We will omit the $R$ when the ring is clear from context.

Given a ring $R$, the {\sl right big finitistic dimension}, $\Fdim R$, is the supremum of the projective dimension of right $R$-modules with finite projective dimension and the {\sl right big weak finitistic dimension}, $\Fwdim R$ is the supremum of the flat dimension of right $R$-modules with finite flat dimension. The {\sl right little finitistic dimension}, f.dim $R$, is the supremum of the projective dimension of right $R$-modules in $\modr R$ with finite projective dimension. 
\\

For any class $\C$ of modules we recall the notion of
a $\C$-{\sl precover}, a {\sl special} $\C$-{\sl precover} and of a $\C$-{\sl cover} (see \cite{Xu}).
\begin{defn} Let $\C$ be a class of modules, $M$ a right $R$-module and $C\in \C$. A homomorphism $\phi\in \Hom_R(C, M)$ is called a
$\C$-{\sl precover} (or right approximation) of $M$ if for every homomorphism $f '\in \Hom_R(C', M)$ with $C'\in \C$ there exists a homomorphism
$f\colon C'\to C$ such that $f '= \phi f$.

A $\C$-precover, $\phi\in \Hom_R(C, M)$ is called a $\C$-{\sl cover} (or a minimal right approximation) of $M$ if for every endomorphism $f$ of $C$ such that  $\phi=
\phi f$,
$f$ is an automorphism of $C$. So a $\C$-cover is a minimal version of a $\C$-precover.

A $\C$-precover $\phi$ of $M$ is said to be {\sl special} if $\phi$ is an epimorphism and $\Ker \phi\in \C^{\perp_1}$.
\end{defn}
The notions of $\C$-{\sl preenvelope} (left approximations), {\sl special} $\C$-{\sl preenvelope} and $\C$-{\sl envelope} (minimal left approximations) are defined dually.

 The relation between $\C$-precovers and $\C$-covers is provided by the following results due to Xu.
\begin{prop}\label{P:Xu} \cite[Corollary 1.2.8]{Xu} Let $\C$ be a class of modules and assume that a module $M$ admits a $\C$-cover. Then a $\C$-precover
$\phi\colon C\to M$ is a
$\C$-cover if and only if $\Ker \phi$ does not contain any non-zero direct summand of $C$.
\end{prop}
A class $\C$ of $R$-modules is called {\sl covering} ({\sl precovering, special precovering}) if every module admits a $\C$-cover ($\C$-precover, special $\C$-precover).

In approximation theory of modules, one is interested in when certain classes provide minimal approximations. Enochs and Xu ~ \cite[Theorem 2.2.8]{Xu},  proved that a precovering class closed under direct limits is covering.
Enochs posed the question to see if the closure under direct limits of a class $\C$ is a necessary condition for the existence of $\C$-covers. 
Our aim is to investigate this problem for the class $\clP_1(R)$.\\%

We consider precovers and preenvelopes for particular classes of modules, that is classes which form a cotorsion pair.

A pair of classes of modules $(\A, \B)$ is a {\sl%
cotorsion pair} provided that $\A = {}^{\perp_1} 
\B$ and $
\B =\A^{\perp_1}$. A cotorsion pair $(\A, \B)$ is called {\sl hereditary} if $\A = {}^{\perp} \B$ and $\B = \A^{\perp}$. 

A cotorsion pair $(\A, \B)$ is {\sl complete} provided that every $R$-module $M$ admits a {special $\B$-preenvelope} or equivalently, every $R$-module $M$ admits a {special $\A$-precover} (\cite{Sal}).

A hereditary cotorsion pair $(\A, \B)$ is of {\sl finite type} if there is a set $\clS\subseteq\modR$ such that $\clS^\perp=\B.$

Examples of complete cotorsion pairs in $\ModR$ include $(\clP_0, \ModR)$, $(\F_0, \F_0^\perp)$,  $(\clP_1(R), \clP_1(R)^\perp)$ (see \cite[Theorem 8.10]{GT12}).\\

A useful result for covers is given by the following.
 \begin{prop}\label{P:A-covers} 
 Let $(\A, \B)$ be a complete cotorsion pair in $\ModR$.  Assume that 
$ A\overset{\phi}\to M\to 0$ is an $\A$-cover of the $R$-module $M$. Let $\alpha$ be an automorphism of $M$ and let $\beta $ be an endomorphism of $A$ such that $\phi\beta=\alpha\phi$. Then $\beta$ is an automorphism of $A$.
\end{prop}
\begin{proof}
By the Wakamatsu Lemma (see \cite[Lemma 2.1.1]{Xu}) $\phi$ gives rise to an exact sequence 
$0\to B\overset{\mu}\to A\overset{\phi}\to M\to 0$ with $B\in \B$. Since $\alpha$ is an automorphism of $M$, it is immediate to see that $\Ker \alpha \phi \cong B \in \B$ and 
that $0\to B\to A\overset{\alpha\phi}\to M\to 0$ is an $\A$-cover of $M$.

Let $\beta$ be as assumed and consider an endomorphism $g$ of $A$ such that $\alpha\phi g=\phi$. Then $\phi\beta g=\phi$ and thus $\beta g$ is an automorphism of $A$, since $\phi$ is an $\A$ cover of $M$. This implies that $\beta$ is an epimorphism. 

To see that $\beta$ is a monomorphism, note that $\alpha \phi g \beta =  \phi \beta g \beta = \phi \beta = \alpha \phi$, thus by the cover property of $\alpha \phi$, $g \beta$ is an automorphism, thus $\beta$ is an automorphism as required.
\end{proof}
\begin{cor}\label{C:cov-of-localisation}
Let $R$ be a commutative ring, $R[S^{-1}]$ be the localisation of $R$ at a multiplicative subset $S$. Let $(\A, \B)$ be a cotorsion pair in $\ModR$. Suppose 
\[
(\ast)\quad 0 \to  B \to A  \overset{\phi} \to M\to 0
\]
 is an $\A$-cover of an $R[S^{-1}]$-module $M$.\\Then $(\ast)$ is a short exact sequence in $\ModR[S^{-1}]$.
\end{cor}
\begin{proof}
Let $M$ and $\phi$ be as assumed. The multiplication by an element of $S$ is an automorphism of $M$. Therefore by Proposition~\ref{P:A-covers}, $A$ is an $R[S^{-1}]$-module. Thus the sequence $(\ast)$ is an exact sequence in $\ModR[S^{-1}]$ since $R \to R[S^{-1}]$ is a ring epimorphism, hence the embedding $\ModR[S^{-1}] \to \ModR$ is fully faithful.
\end{proof}
\section{The direct limit closure of $\clP_1(R)$}\label{S:limP_1}
From now on, $R$ will always be a ring such that $\Sigma$, the set of regular elements of $R$, satisfies both the left and right Ore conditions. The {\sl classical ring of quotients} of $R$, denoted $Q = Q(R)$ is the ring $R[\Sigma^{-1}] = [\Sigma^{-1}] R$ which is flat both as a right and a left $R$-module. 
Additionally, we recall that an ideal $I$ of $R$ is called {\sl regular} if $I$ contains a regular element of $R$, that is $I \cap \Sigma \neq \emptyset$.
 
 Recall that $\clP_1(R)$ denotes the class of right $R$-modules with projective dimension at most $1$ and $\clPmR$ is the set $\clP_1(R)\cap \modR$. 

A {\sl $1$-cotilting class of cofinite type} is the $\Tor^R_1$-orthogonal of a set of modules in $\clPmR$. Thus the {\sl minimal} $1$-cotilting class of cofinite type is $\clPmR^\intercal$, which we will denote by $\C(R)$. \\

The purpose of this section is to describe the class $\varinjlim \clP_1(R)$ generalising a result in \cite[Theorem 6.7 (vi)]{BH09} which was proved under the assumption $\fdim Q =0$.
We begin by recalling the following corollary, which states that one can consider only the finitely presented modules in $\clP_1(R)$ to find its direct limit closure.
\begin{thm}\cite[Corollary 9.8]{GT12}\label{T:lim-P1-P1modR}
Let $R$ be a ring. Then $\varinjlim \clP_1(R) = \varinjlim \clPmR= {}^\intercal(\clPmR^\intercal)$ and $\clPmR^\intercal = \clP_1(R)^\intercal=\big(\varinjlim \clP_1(R)\big)^\intercal.$
\end{thm}
Following the nomenclature of \cite{BH09}, $\D$ will denote the class $\{D\in{}\ModR \mid \Ext^1_R(R/rR, D)=0,\, r \in{}\Sigma \}$ of {\sl divisible} right $R$-modules and $\T\F$ will denote the class $\{ N\in \RMod \mid \Tor^R_1(R/rR,\,  N)=0,\, r \in \Sigma \}$ of {\sl torsion-free} left $R$-modules. The analogous statements hold for the divisible modules in $\RMod$ and the torsion-free modules in $\ModR$.

By \cite[Lemma 5.3]{BH09}, for a ring $R$ with a classical ring of quotients $Q$ and every torsion-free left $R$-module ${}_RN$, $\Tor^R_1(Q/R, N)=0$. Analogously, for every torsion-free right $R$-module $N_R$, $\Tor^R_1(N, Q/R)=0$. 

Additionally, \cite[Lemma 6.2]{BH09} establishes that for  a ring $R$ with classical ring of quotients $Q$, a right $Q$-module $V$ is in $\clP_1(Q)$ if and only if there is a right $R$-module $M$ in $\clP_1(R)$ such that $V=M\otimes_RQ$. 

The following lemma is a sort of analogue to \cite[Lemma 6.2]{BH09} for the finitely presented case.
\begin{lem}\label{L:fp-P1-Q}
Let $R$ be a ring with classical ring of quotients $Q$. If $C_Q \in \clPmQ$, then there exist $P_Q \in \clP_0(\modQ)$ and $N_R \in \clPmR$ such that $C_Q \oplus P_Q \cong N \otimes_R Q$.
\end{lem}
\begin{proof}
 The argument follows from the proofs of ~\cite[Lemma 6.4]{BH09} and \cite[Lemma 6.2]{BH09}, which we reiterate here for completeness.
 
Take $C_Q \in \clPmQ$. Then by \cite[Lemma 6.4]{BH09}, there exists a $P_Q \in \clP_0(\modQ)$ and a short exact sequence 
\[
0 \to Q^m \overset{\mu}\to Q^n \to C_Q \oplus P_Q \to 0
\]

Let $(d_1, \dots, d_m)$  be the canonical basis of the right
$Q$-free module $Q^{m}$. The monomorphism $\mu$ is represented
by a column-finite matrix $A'$ with entries in $Q=R[\Sigma^{-1}]$
acting as left multiplication on the basis elements $d_i$. 

Change the basis $(d_1, \dots, d_m)$ to the basis $(r_id_i\colon 1\leq i\leq m)$ where $r_1\in \Sigma$ is a common denominator of the elements of the $i^{\rm{th}}$ column of $A'$, so that the morphism $\mu$ can be represented
by a column-finite matrix $A$ with entries in $R$. As $R$ is contained in
$Q$, we get the short exact sequence
\[0\to R^m\overset{\nu}\to R^n\to\Coker \nu\to 0,\]
where the map $\nu$ is represented by the matrix $A$. Therefore $\nu \otimes_R \id_Q = \mu$ so $\Coker \nu \in \clPmR$ is the desired $N$.
\end{proof}
%
Recall that $\C(R)=\clPmR^\intercal=\clP_1(R)^\intercal=\big(\varinjlim \clP_1(R)\big)^\intercal$.
\begin{lem}\label{L:min-cotilt-Q}
Let $R$ be a ring with classical ring of quotients $Q$. Then the following hold. 
\begin{enumerate}
\item[(i)] $\C(R) \cap \QMod = \C(Q)$. 
\item[(ii)] If $Z \in \C(R)$, then $Q \otimes_R Z \in \C(Q)$. 
\end{enumerate}
\end{lem}
\begin{proof}
(i) By well-known homological formulas, the flatness of $Q$ implies that for each $M \in \ModR$ and $N \in \QMod$, there is the following isomorphism.
\[\Tor^R_1(M, N) \cong \Tor^Q_1(M \otimes_R Q, N)\]
Suppose $M \in \clP_1(R)$. Then $N \in \C(R) \cap \QMod$ if and only if the left-hand side in the above isomorphism vanishes. On the other hand, $N{}\in{} \C(Q)$ \iff $N \in \clP_1(Q)^\intercal$ which in view of \cite[Lemma 6.2]{BH09} amounts to the right-hand side in the above isomorphism vanishing. Therefore $N \in \C(R) \cap \QMod$ \iff $N \in \C(Q)$, which proves $\C(R) \cap \QMod = \C(Q)$.

For (ii), we first note that $\C(R)$ is closed under direct limits as $\Tor$ commutes with direct limits. 
%
As $Q$ is both left and right flat, one can write $Q$ as a direct limit of finitely generated free right $R$-modules $\varinjlim_\alpha R^{n_\alpha}$. Fix a $Z \in \C(R)$. Then $Q \otimes_R Z \cong \varinjlim_\alpha Z^{n_\alpha}$ which must be in $\C(R)$ as $\C(R)$ is closed under direct limits. Moreover, $Q \otimes_R Z \in \QMod$, thus since $\C(R) \cap \QMod = \C(Q)$ from (i), $Q \otimes_R Z \in \C(Q)$.
\end{proof}
\begin{prop}\label{P:P1-F1}
Let $R$ be a ring with classical ring of quotients $Q$. Then  
\[ \varinjlim \clP_1(R) = \F_1(R) \cap {}^{\intercal} \C (Q),\]
where ${}^{\intercal} \C (Q)$ represents the left $\Tor^R_1$-orthogonal of $\C(Q)$ in $\ModR$. 

In particular, if $\fdim Q =0$, $\varinjlim \clP_1(R) = \F_1(R) \cap {}^{\intercal} \QMod$.
\end{prop}
\begin{proof}
First we suppose that $M \in \varinjlim \clP_1(R)$ and show that $M \in \F_1(R) \cap {}^{\intercal} \C (Q)$. The inclusion $\varinjlim \clP_1(R) \subseteq \F_1(R)$ always holds so it remains to show that $M \in {}^{\intercal} \C (Q)$. By Theorem~\ref{T:lim-P1-P1modR}, if $M \in \varinjlim \clP_1(R)$, then $M \in {}^{\intercal} \C (R)$. As $\C(Q) \subseteq \C(R)$ by Lemma~\ref{L:min-cotilt-Q}(i), it follows that ${}^{\intercal} \C (R) \subseteq {}^{\intercal} \C (Q)$, so $M \in {}^{\intercal} \C (Q)$ as required. 

For the converse, fix $M \in \F_1(R) \cap {}^{\intercal} \C (Q)$. We will show $M \in {}^{\intercal} \C (R)$, thus the conclusion follows as $\varinjlim \clP_1(R) = {}^{\intercal} \C (R)$ by Theorem~\ref{T:lim-P1-P1modR}. 

The class $\C(R)$ is contained in $\{ R/sR \mid s \in \Sigma\}^\intercal$, so it consists of torsion-free left $R$-modules, thus by \cite[Lemma 5.3]{BH09}, $\Tor^R_1(Q/R, N)$ for every $N \in \C(R)$.
Therefore for every $N \in \C(R)$ there is a short exact sequence in $\RMod$, 
\[ 0 \to N \to Q \otimes_R N \to Q/R \otimes_R N \to 0.\]
Apply $M \otimes_R -$ to this sequence to get the following exact sequence 
\[\Tor^R_2(M, Q/R \otimes_R N) \to \Tor^R_1(M, N) \to \Tor^R_1(M, Q \otimes_R N).\]
The left-most term vanishes as $M \in \F_1(R)$ and the right-most term vanishes as $M \in {}^{\intercal} \C (Q)$ and $N \in \C(R)$ implies $Q \otimes_R N \in \C(Q)$ by Lemma~\ref{L:min-cotilt-Q}(ii). Therefore the central term vanishes for every $N \in \C(R)$, that is $M \in {}^{\intercal}\C(R)$.

The final statement follows since, if $\fdim Q=0$, then $\clPmQ = \clP_0(\modQ)$. Therefore $\C(Q)= \clP_0(\modQ)^\intercal= \QMod$.
\end{proof}
%


We denote by $\B(R)$ the right $\Ext_R^1$-orthogonal class $\clP_1(R)^\perp$ in $\ModR$. This notation is particularly useful because we often change between the ring $R$ and its localisation $Q$ and their module categories. 

Recall that by ~\cite[Proposition 4.1]{BH09}, the cotorsion pair $(\clP_1(R),\B(R))$ is of finite type if and only if $\B(R)$ is closed under direct sums.
\\

 {\sl From now on all rings will be commutative}.
 
  When $R$ is commutative, its classical ring of quotients is also called the total ring of quotients. 
We begin with a lemma. 
\begin{lem}\label{L:P1-perp}
Let $R$ be a commutative ring. Then $\B(R) \cap \ModQ = \B(Q)$. 
\end{lem}
\begin{proof}
Fix some $P \in \clP_1(R)$ and $B \in \ModQ$ and consider the natural isomorphism
\[
\Ext^1_R(P, B) \cong \Ext^1_Q(P \otimes_R Q, B)
.\]
If $B \in \B(R)$, then the left-hand side in the isomorphism vanishes, thus $B \in \B(Q)$, as every module in $\clP_1(Q)$ is of the form $P \otimes_R Q$ for some $P \in \clP_1(R)$ by \cite[Lemma 6.2]{BH09}. Conversely, if $B \in \B(Q)$, then the right-hand side vanishes, so we conclude that $B \in \B(R)$.
\end{proof}
Recall that a commutative ring $R$ is {\sl perfect} if and only if $\Fdim R=0$. 
In \cite{BSa} and \cite{FS}, a commutative ring $R$ is called  {\sl almost perfect} if its total ring of quotients $Q$ is a perfect ring and for every regular non-unit element $r$ in $R$, $R/rR$ is a perfect ring.
 
If $R$ is a ring and $\{ a_1, a_2, \dots, a_n, \dots\}$ is a sequence of elements of $R$, a {\sl Bass right $R$-module} is a flat module of the form
\[F=\varinjlim(R\overset{a_1}\to R\overset{a_2}\to R\overset{a_3}\to\cdots).\]
All Bass $R$-modules have projective dimension at most one. Thus the class of Bass $R$-modules is contained in $\F_0(R) \cap \clP_1(R)$.
In \cite{Bass}, Bass noticed that a (not-necessarily commutative) ring $R$ is right perfect if and only if every Bass right $R$-module is projective.

  \begin{prop}\label{P:conj-perfect} Let $R$ be a commutative ring with total ring of quotients $Q$. Consider the following conditions:
\begin{enumerate}
\item[(i)] $\clP_1=\F_1$.
\item[(ii)] $\clP_1$ is closed under direct limits.
\item[(iii)] For every regular non-unit element of $R$, $R/rR$ is a perfect ring.
\item[(iv)]  $R$ is an almost perfect ring.
\end{enumerate}
Then (i) $\Rightarrow$ (ii) $\Rightarrow$ (iii) and (iv) $\Rightarrow$ (iii).

If $\Fwdim Q=0$ then (i) and (ii) are equivalent; if moreover $Q$ is a perfect ring, then all the four conditions are equivalent.
\end{prop}
\begin{proof}
The implications (i) $\Rightarrow$ (ii) and (iv) $\Rightarrow$ (iii) are straightforward.

The implication (ii) $\Rightarrow$ (iii) is a slight generalisation of \cite[Proposition 8.5]{BP1}, where (i) $\Rightarrow$ (iii) is proved. Suppose $\clP_1(R)$ is closed under direct limits, fix an element $r \in \Sigma$ and a Bass $R/rR$-module $N$. We will show that $N$ is projective as an $R/rR$-module, so that we conclude that $R/rR$ is a perfect ring. As $r$ is regular, $R/rR \in \clP_1(R)$ so all projective $R/rR$-modules are in $\clP_1(R)$. $N$ is a flat  $R/rR$-module, so $N \in \varinjlim \clP_0(R/rR) \subseteq \varinjlim \clP_1(R) = \clP_1(R)$, where the last equality holds by assumption. Moreover, $N \in \clP_1(R/rR)$ so we can apply the Change of Rings Theorem obtaining
\[
\pdim_R N = \pdim_{R/rR} N +1
.\] It follows that $\pdim_{R/rR} N=0$, as required. 

 If $\Fwdim Q=0$, then (i) and (ii) are equivalent by \cite[Corollary 6.8]{BH09}.
 If $Q$ is perfect ring, the equivalence of the four conditions is proved in \cite[Theorem 6.1]{FS} or \cite[Proposition 8.7]{BP1}.
\end{proof}
The example below shows that the condition  $\clP_1(R)=\F_1(R)$ doesn't imply that the total ring of quotients $Q$ is a perfect ring.
\begin{expl}\label{Ex:berg}
In \cite[5.1]{Ber} it is shown that there is a totally disconnected topological space $X$ whose ring $R$ of continuous functions is Von Neumann regular and hereditary. Moreover, every regular element of $R$ is invertible. Hence $R$ coincides with its own total ring of quotients and $\clP_1(R)=\F_0(R)=\Modr R$, but $R$ is not perfect, since it is not semisimple.
Moreover, since $R$ is hereditary, the class $\clP_1(R)^\perp$ coincides with the class of injective $R$-modules which is not closed under direct sums, since $R$ is not noetherian.\end{expl}

\section{Properties of some classes of commutative rings}\label{S:properties}
We recall now  the characterisations of some classes of commutative rings. 

Recall that a commutative ring $R$ is {\sl semihereditary} if every finitely generated ideal is projective.

By \cite[Corollary 4.2.19]{Glaz89} $R$ is semihereditary if and only if  $Q(R)$ is Von Neumann regular and for every prime ideal $\p$, $R_\p$ is a valuation domain. In particular, by \cite[Theorem 4.2.2]{Glaz89}, $R$ is reduced, that is $R$ contains no nilpotent elements.

The following proposition is modelled on Cohen's Theorem, which states that if all prime ideals are finitely generated, then all ideals are finitely generated, see for example \cite[Theorem 8]{Kap} or \cite[Theorem 3.4]{Na}. In the following we consider only the regular ideals.
\begin{prop}\label{P:prime-reg-fg}
Let $R$ be a commutative ring. If every regular prime ideal is finitely generated, then every regular ideal is finitely generated. 
\end{prop}
\begin{proof} The arguments we use are exactly as in  Cohen's Theorem, but we repeat the proof to outline the steps in which regularity is used.

Let $\Theta$ be the collection of regular ideals which are not finitely generated with a partial order by inclusion. Assume, by way of contradiction  that $\Theta$ is not empty. 
Let $\Phi$ be a totally ordered subset of $\Theta$ and let $I:= \bigcup_{J \in \Phi}J$. We  claim that $I$ is in $\Theta$ so that it is an upper bound of $\Phi$ in $\Theta$. Clearly $I$ contains a regular element, so it remains to show that $I$ is not finitely generated. Suppose for contradiction that $I$ has a finite set of generators $\{ a_1, \dots, a_n \}$. Then there exists a $J_0 \in \Phi$ such that $I=< a_1, \dots, a_n > \subseteq J_0 \subseteq I $, therefore $J_0$ is finitely generated which is a contradiction. 

Thus by Zorn's Lemma, $\Theta$ has a maximal element. We will show that such a maximal element is prime, obtaining a contradiction. Fix a maximal element $L$ of $\Theta$, and suppose it is not prime, that is there exist two elements $a, b \in R \setminus L$ such that $ab \in L$. Then both $L + aR$ and $L + bR$ strictly contain $L$, so they are both finitely generated. Therefore, there exist $x_1, \dots, x_n \in L$ and $y_1, \dots, y_n \in R$ so that $
\{ x_1 +ay_1, \dots, x_n +ay_n\}
$ is a generating set of $L+aR$. 

Consider the ideal $H := (L: a)= \{ r \mid ra \in L\}$. Then $L \subsetneq L+bR \subseteq H$, therefore also $H$ is finitely generated, and so also $aH$ is finitely generated. We now will show that $L = <x_1, \dots, x_n>+aH$ so $L$ is finitely generated. An element $r \in L \subsetneq L+aR$ can be written as follows.
\[
r = s_1(x_1 + a y_1) + \cdots + s_n(x_n + a y_n) = \Sigma_i s_i x_i + a \big( \Sigma_i s_i y_i \big)
\]

$ \Sigma_i s_i y_i \in H$ as $a \big( \Sigma_i s_i y_i \big) = r -  \Sigma_i s_i x_i \in L$.\\ Therefore $L \subseteq <x_1, \dots, x_n>+aH$. The converse inclusion is clear, so $L = <x_1, \dots, x_n>+aH$ which implies that $L$ is finitely generated as $H$ is, a contradiction. Therefore $L$ is prime, and so by the assumption that every prime ideal is finitely generated, $\Theta$ must be empty.  
\end{proof}
%
\begin{lem}\label{L:sh-reg-ideals}
Let $R$ be a reduced commutative ring with total ring of quotients $Q$ of Krull dimension $0$ (for example when $R$ is semihereditary). An ideal $I$ of $R$ is contained in a minimal prime ideal of $R$ if and only if $I$ is not regular.
\end{lem}
\begin{proof}
First suppose that $I \subseteq \p$ where $\p$ is a minimal prime ideal of $R$.  Since $R$ is reduced it is an easy exercise to show that the set of zero divisors of $R$ coincides with the union of the minimal prime ideals (see \cite[Exercise 2.2.13, page 63]{Kap}).

For the converse, suppose that $I$ is not regular. Let $L$ be an ideal maximal with respect to the properties $I \subseteq L$, $L \cap \Sigma = \emptyset$. Then as in \cite[Theorem 1.1]{Kap}, $L$ is a prime ideal. 
Assume there is a prime ideal $\p \leq L$. Since $\p$ and $L$ are not regular, $\p Q\leq LQ$.
 By assumption $Q$ has Krull dimension $0$, hence $\p Q=LQ$, which implies $\p =L$, that is $L$ is a minimal prime.
\end{proof}
We also have the following proposition from a paper of Vasconcelos.
\begin{prop}\cite[Proposition 1.1]{Vas}\label{P:proj-ideals}
Let $R$ be a commutative ring with a projective ideal $I$. If $I$ is not contained in any minimal prime ideal it is finitely generated.
\end{prop}
%
\section{When $\clP_1(R)$ is covering for commutative rings}\label{S:sh-P1-cov}
In this section we collect some facts about when $\clP_1(R)$ is a covering class, and in particular we state some consequences for the total ring of quotients $Q$ of $R$ or for localisations of $R$.

\begin{lem}\label{L:local-P1-cov}
Let $R$ be a commutative ring and suppose $\clP_1(R)$ is covering in $\ModR$. Then the following hold.
\begin{enumerate}
\item[(i)] $\clP_1(R) \cap \ModQ = \clP_1(Q)$
\item[(ii)] $\clP_1(Q)$ is covering in $\ModQ$.
\end{enumerate}
\end{lem}
\begin{proof}
(i) The inclusion $\clP_1(R) \cap \ModQ \subseteq \clP_1(Q)$ is clear. For the converse, take $M \in \clP_1(Q)$ and consider a $\clP_1(R)$-cover of $M$, 
$0 \to B \to A \overset{\phi}\to M \to 0.$
Then $B\in \B(R)$ and by Corollary~\ref{C:cov-of-localisation}, $A, B \in \ModQ$. From Lemma~\ref{L:P1-perp}, $B \in \B(Q)$, so $\phi$ splits and $M$ must be in $\clP_1(R)$.

(ii) Fix a $Q$-module $M$ and let $0 \to B \to A \overset{\phi}\to M \to 0 $ be a $\clP_1(R)$-cover of $M$. Then again by Corollary~\ref{C:cov-of-localisation}, by (i) and by Lemma~\ref{L:P1-perp}, $A \in \clP_1(Q)$ and $B \in \B(Q)$ so $\phi$ must be a $\clP_1(Q)$-precover of $M$. To see that is a cover, any endomorphism $f$ of $A$ in $\ModQ$ is also a homomorphism in $\ModR$, therefore by the minimality property of $\phi$ as a $\clP_1(R)$-cover, $\phi$ is also a $\clP_1(Q)$-cover.
\end{proof}
We now find some consequences of when $\clP_1(R)$ is covering for localisations of $R$.
\begin{lem}\label{L:Matlis} Let $R$ be a commutative ring and let $S$ be a multiplicative subset of $R$. Then the following hold.
\begin{enumerate}
\item[(i)] If $R[S^{-1}]$ has a $\clP_1(R)$-cover, then $\pdim_RR[S^{-1}] \leq 1$. 
\item[(ii)] If $M\in \clP_1(R)$ and $M\otimes_RR[S^{-1}]$ admits a $\clP_1(R)$-cover, then $\pdim_R(M\otimes_RR[S^{-1}])\leq 1$.
\item[(iii)] Suppose $\clP_1(R)$ is covering. Let $S$,  $T$ be  multiplicative systems of $R$ with $S\subseteq \Sigma$.
 Then $\pdim_R\dfrac{R[S^{-1}]\otimes_R R[T^{-1}]}{R[T^{-1}]}\leq 1$.
\end{enumerate}
\end{lem}
\begin{proof}
(i) Let the following be a $\clP_1(R)$-cover of $R[S^{-1}]$.
\begin{equation}\label{eq:book}
 0\to Y\to A\to R[S^{-1}]\to 0
 \end{equation} 
By Corollary~\ref{C:cov-of-localisation}, both $A$ and $Y$ are $R[S^{-1}]$-modules as well. Thus (\ref{eq:book}) is an exact sequence of $R[S^{-1}]$-modules, hence it splits. We conclude that $R[S^{-1}]$ is a direct summand of $A$ as an $R$-module, hence $\pdim R[S^{-1}]\leq1$.

(ii)Suppose $M\in \clP_1(R)$ and let the following be a $\clP_1(R)$-cover of $M\otimes_RR[S^{-1}]$.
\begin{equation}\label{eq:book2}
 0\to Y\to A\to M\otimes_RR[S^{-1}]\to 0
 \end{equation} 
As in (i), we conclude by Corollary~\ref{C:cov-of-localisation} that the sequence (\ref{eq:book2}) is in $\ModR[S^{-1}]$. Thus, $\Ext^1_{R[S^{-1}]}(M\otimes_RR[S^{-1}], Y)\cong \Ext^1_R(M, Y)=0$ since $Y\in \B(R)$ and $M\in \clP_1(R)$. Therefore $M\otimes_RR[S^{-1}]$ is a summand of $A$, hence it has projective dimension at most one.

(iii)Suppose $\clP_1(R)$ is covering in $\ModR$. By (i) $\pdim_RR[S^{-1}]/R\leq 1$ as $S \subseteq \Sigma$ so $R \to R[S^{-1}]$ is a monomorphism. Hence (iii) follows by (ii).
\end{proof}

\begin{rem}\label{R:ft-comm-dom} If the class $\clP_1(R)^\perp$ is closed under direct sums, then by \cite[Theorem 1.3.6]{AST17}, or by \cite[Remark 7.4]{BPS}, $\clP_1(R)$ is covering if and only if it is closed under direct limits. 

When $R$ is a commutative domain, then by ~\cite[Corollary 8.1]{BH09}, the class $\clP_1(R)^\perp$ is closed under direct sums (and it coincides with the class of divisible modules) and the four conditions in Proposition~\ref{P:conj-perfect} are equivalent.

\end{rem}


\section{When $\clP_1(R)$ is covering for commutative semihereditary rings}\label{S:sh-comm}

In this section we restrict to looking at commutative semihereditary rings. First we characterise the class of semihereditary commutative rings such that the cotorsion pair $(\clP_1(R), \B(R))$ is of finite type, or equivalently such that $\B(R)$ is closed under direct sums or equivalently $\B(R)= \clP_1(\modR)^\perp$ (see \cite[Lemma 4.1]{BH09}).  
\begin{prop}\label{P:not-finite-type} Let $R$ be a commutative semihereditary ring. The following are equivalent.
\begin{enumerate}
\item[(i)] The cotorsion pair $(\clP_1(R), \B(R))$ is of finite type.
\item[(ii)] The total quotient ring $Q$ of $R$ is semisimple.
\item[(iii)] $R$ is a finite direct product of Pr\"ufer domains.
\end{enumerate}
\end{prop}
\begin{proof} (i) $\Rightarrow$ (ii). By assumption, $\B(R)$ is closed under direct sums, thus also $\B(Q)$ is closed under direct sums by Lemma~\ref{L:P1-perp}. Thus $\B(Q)= \clP_1(\modQ)^\perp$.
By assumption $Q$ is Von Neumann regular, hence $\clP_1(\modQ)=\clP_0(\modQ)$ as every $Q$-module is flat. Therefore $\B(Q)=\ModQ$ and we conclude that $\clP_1(Q)=\clP_0(Q)$, that is $Q$ is a perfect ring. Thus, $\ModQ=\F_0(Q)=\clP_0(Q)$, which means that $Q$ is semisimple.
 
 (ii) $\Rightarrow$ (iii). Follows by \cite[Corollary, page 117]{Endo} as a semihereditary ring $R$ has weak global dimension at most one.
 
 (iii) $\Rightarrow$ (i). Obvious by Remark~\ref{R:ft-comm-dom}.
 \end{proof}
By the above proposition we conclude that the class of semihereditary commutative rings such that the cotorsion pair $(\clP_1(R), \B(R))$ is of not finite type is rather big.

We now begin our investigation of the cotorsion pair $(\clP_1(R), \B(R))$ for any commutative semihereditary ring $R$.

The following holds also for not-necessarily commutative rings.
\begin{lem}\label{L:sh-limP1}
Suppose $R$ is a semihereditary ring. Then $\clP_1(R)$ is closed under direct limits if and only if $R$ is hereditary.
\end{lem}
\begin{proof}
Sufficiency is clear, since $R$ is hereditary if and only if $\clP_1(R)=\ModR$.
The converse follows immediately since every $R$-module is a direct limit of finitely presented modules and if $R$ is semihereditary, every finitely presented module is in $\clP_1(R)$.
\end{proof}
We first consider the case of a Von Neumann regular commutative ring.
%
\begin{prop}\label{P:vnr-P1} Let $R$ be a Von Neumann regular commutative ring. Then $\clP_1(R)$ is covering if and only if $R$ is a hereditary ring.
\end{prop}
\begin{proof} $R$ is semi-hereditary, so it remains to show that every infinitely generated ideal $I$ of $R$  is projective. Consider a $\clP_1(R)$-cover of $R/I$ 
 \[(\ast)\quad 0\to B\to A\to R/I\to 0.\]
 The ideal $I$ is the sum of its finitely generated ideals which are all of the form $eR$, for some idempotent element $e\in R$. For every idempotent element $e\in I$, we have $Ae\subseteq B$, hence by Proposition~\ref{P:Xu} $Ae=0$. We conclude that $AI=0$.
 On the other hand, $A=B+xR$ for some element $x\in A$ such that $xR\cap B=xI$. Thus $B\cap xR=0$, since $AI=0$ and we infer that the sequence $(\ast)$
 splits, thus p.dim $R/I\leq 1$ and $I$ is projective.
 \end{proof}

We pass now to the case of semihereditary commutative rings.
%
\begin{lem}\label{L:Rp-P1-covering}
Let $R$ be a commutative semihereditary ring. Then the following statements hold. 
\begin{enumerate}
\item[(i)] For every $M \in \clP_1(R)$, $M \otimes_R R_\p \in \clP_1(R_\p) $
\item[(ii)] For every $N \in \B(R)$, $N \otimes_R R_\p \in \B(R_\p)$.
\item[(iii)] If $\clP_1(R)$ is covering then $\clP_1(R_\p)$ is covering.
\end{enumerate}
\end{lem}
\begin{proof}
(i) Clear as $R_\p$ is flat. \\
(ii) As $R_\p$ is a commutative domain, the cotorsion pair $(\clP_1(R_\p), \B(R_\p))$ is of finite type, and $\B(R_\p)$ coincides with the class of divisible modules by ~\cite[Theorem 7.2]{BH09}. Thus it is sufficient to show that for every $N \in \B(R)$, $\Ext^1_{R_\p}(R_\p/aR_\p, N \otimes_R R_\p)=0$ for each $a \in R_\p$. Without loss of generality, we can assume that $a \in R$. As $R$ is commutative and $R/aR \in \clPmR$ since $R$ is semihereditary, there is the following isomorphism.
\[
 \Ext^1_R(R/aR, N)_\p \cong \Ext^1_{R_\p}(R_\p/aR_\p, N_\p) 
 \]
 As $R/aR \in \clP_1(R)$, the left-hand side vanishes as required. \\
(iii) Let  $(\ast\ast)\ 0\to B\to A \overset{\phi}\to M \to 0
$ be a $\clP_1(R)$-cover of $M \in \ModR_\p$.
Then $A, B \in \ModR_\p$ by Corollary~\ref{C:cov-of-localisation}, and by (ii), $B \in \B(R_\p)$. Therefore, $(\ast\ast)$ is also a $\clP_1(R_\p)$-precover of $M$ in $\ModR_\p$. Moreover, since any $R_\p$-module homomorphism is also an $R$-module homomorphism, $(\ast\ast)$ is a $\clP_1(R)$-cover of $M$. 
\end{proof}
\begin{lem}\label{L:locals-dvrs}
Let $R$ be a commutative semihereditary ring such that $\clP_1(R)$ is covering. Then for each prime $\p$, the ring $R_\p$ is a discrete valuation domain.
Moreover, $\ModR_\p \subseteq \clP_1(R)$. 

As a consequence, every maximal ideal $\m$ in $R$ is projective. 
\end{lem}
\begin{proof}
 First note that as $R_\p$ is a valuation domain, it is also semihereditary. By Lemma~\ref{L:Rp-P1-covering}(iii), $\clP_1(R_\p)$ is covering. Therefore by Remark~\ref{R:ft-comm-dom}, $\clP_1(R_\p)$ is closed under direct limits, so 
 by Lemma~\ref{L:sh-limP1} we conclude that $R_\p$ is hereditary, hence a discrete valuation domain.
 
 To see that $\ModR_\p \subseteq \clP_1(R)$, let $ 0\to B\to A \overset{\phi}\to M \to 0$ be a $\clP_1(R)$-cover of an $R_\p$-module $M$. 
As in the proof of Lemma~\ref{L:locals-dvrs}, the sequence is
also a $\clP_1(R_\p)$-cover of $M$ in $\ModR_\p$. We have just shown that $R_\p$ is hereditary, so $M \in \clP_1(R_\p)$, hence $\phi$ is an isomorphism. Since $A \in \clP_1(R)$, also $M \in \clP_1(R)$ for any $R_\p$-module $M$. 
 
 For the second statement, let $\m$ be a maximal ideal of $R$. Once one observes that $R/ \m$ is an $R_\m$-module, it follows that $\m$ is projective as $\ModR_\m \subseteq \clP_1(R)$ by what is proved above. 
 \end{proof}
\begin{lem}\label{L:reg-prime-are-max}
Let $R$ be a commutative semihereditary ring such that $\clP_1(R)$ is covering. Then every regular prime ideal is maximal.
\end{lem}
\begin{proof}
Take $\p$ to be a regular prime ideal of $R$. Then by Lemma~\ref{L:sh-reg-ideals}, $\p$ cannot be minimal. Fix a maximal ideal $\m$ such that $\p \subseteq \m$. Then by Lemma~\ref{L:locals-dvrs} in the localisation $R_\m$, there are exactly two prime ideals, $0$ and $\m_\m$, which are in bijective correspondence with the prime ideals of $R$ contained in $\m$. As $\p$ cannot be minimal, one concludes that $\p = \m$, therefore $\p$ is maximal. 
\end{proof}
The following corollary follows easily.
\begin{cor}\label{C:reg-max-fg}
Let $R$ be a commutative semihereditary ring such that $\clP_1(R)$ is covering. Then every regular prime (hence maximal) ideal is finitely generated.
\end{cor}
\begin{proof}
Let $\p$ be a regular prime ideal of $R$. By Lemma~\ref{L:reg-prime-are-max} and Lemma~\ref{L:locals-dvrs} $\p$ is a projective ideal. Hence by Lemma~\ref{L:sh-reg-ideals} and Proposition~\ref{P:proj-ideals} $\p$ is finitely generated. 
\end{proof}
%
We will use the following characterisation of hereditary rings.
\begin{thm}\cite[Corollary 4.2.20]{Glaz89},\cite[Theorem 1.2]{Vas}\label{T:hered-rings}
Let $R$ be a commutative ring. Then $R$ is hereditary if and only if $Q(R)$ is hereditary and any ideal of $R$ that is not contained in any minimal prime ideal of $R$ is projective. 
\end{thm}
We now can state the main result of this paper.
\begin{thm}\label{T:sh-P1-cov-lim}
Let $R$ be a commutative semihereditary ring such that $\clP_1(R)$ is covering. Then $R$ is hereditary. Therefore $\clP_1(R)$ is closed under direct limits.
\end{thm}
\begin{proof}
We use Theorem~\ref{T:hered-rings} to show that $R$ must be hereditary. First we show that the classical ring of quotients, $Q$, is hereditary. From Lemma~\ref{L:local-P1-cov} and the assumption that $\clP_1(R)$ is covering, we know that $\clP_1(Q)$ is covering. Additionally, by \cite[Corollary 4.2.19]{Glaz89} $Q$ is Von Neumann regular. Therefore, $Q$ must be hereditary by Proposition~\ref{P:vnr-P1}. 

Now we show that any ideal not contained in a minimal prime ideal is projective. By Lemma~\ref{L:sh-reg-ideals}, it is enough to show that any regular ideal is projective, which follows if any regular ideal is finitely generated as $R$ is semihereditary. By Proposition~\ref{P:prime-reg-fg}, it is sufficient to show that the regular prime ideals are finitely generated, which follows from Corollary~\ref{C:reg-max-fg}. We conclude that all ideals not contained in a minimal prime ideal are finitely generated, and hence are projective as $R$ is semihereditary.
\end{proof}

\bibliographystyle{alpha}
\bibliography{references}
\end{document}